\documentclass[reqno,12pt]{amsart}
\usepackage[all,line,poly,rotate,ps,dvips]{xy}
\SelectTips{cm}{}
\usepackage{amssymb,eucal,graphicx}
\usepackage{enumerate}
\numberwithin{equation}{section}
\newtheorem{theorem}[equation]{Theorem}

\newtheorem{corollary}[equation]{Corollary}
\newtheorem{claim}[equation]{Claim}
\newtheorem{lemma}[equation]{Lemma}
\newtheorem{proposition}[equation]{Proposition}

\newtheorem{ass}[equation]{Assumptions}

\newtheorem{constr}[equation]{Construction}
\theoremstyle{definition}

\newtheorem{definition}[equation]{Definition}
\newtheorem{remark}[equation]{Remark}
\newtheorem{conjecture}[equation]{Conjecture}

\theoremstyle{remark}

\newcommand{\ZZ}{\mathbb{Z}}
\newcommand{\G}{\mathbb{G}}

\newcommand{\Pp}{\mathbb{P}}

\newcommand{\Ff}{\mathcal{F}}
\newcommand{\Pm}{\mathcal{P}}

\newcommand{\HH}{\mathcal{H}}

\newcommand{\LL}{\mathcal{L}}
\newcommand{\N}{\mathcal{N}}

\newcommand{\Oc}{\mathcal{O}}

\begin{document}
\title[Brill-Noether theory and non-special scrolls]%
{Brill-Noether theory and non-special scrolls}
\author{Alberto Calabri, Ciro Ciliberto, Flaminio Flamini, Rick Miranda}

\email{calabri@dmsa.unipd.it}
\curraddr{Dipartimento di Metodi e Modelli Matematici per le
Scienze Applicate, Universit\`a degli Studi di Padova\\
Via Trieste, 63 - 35121 Padova \\Italy}
\email{cilibert@mat.uniroma2.it} \curraddr{Dipartimento di
Matematica, Universit\`a degli Studi di Roma Tor Vergata\\ Via
della Ricerca Scientifica - 00133 Roma \\Italy}
\email{flamini@mat.uniroma2.it} \curraddr{Dipartimento di
Matematica, Universit\`a degli Studi di Roma Tor Vergata\\ Via
della Ricerca Scientifica - 00133 Roma \\Italy}
\email{Rick.Miranda@ColoState.Edu} \curraddr{Department of
Mathematics, 101 Weber Building, Colorado State University
\\ Fort Collins, CO 80523-1874 \\USA}

\thanks{
\\
The first three authors are members of G.N.S.A.G.A.\ at
I.N.d.A.M.\ ``Francesco Severi''.}

\begin{abstract} In this paper we study  the
Brill-Noether theory of sub-line bundles of a general, stable rank-two vector bundle
on a curve $C$ with general moduli.  We relate this theory
to the geometry of unisecant curves on smooth, non-special scrolls with hyperplane sections isomorphic
to $C$. Most of our results are based on degeneration techniques.
\end{abstract}
\maketitle
\noindent

\section{Introduction}\label{S:Intro}

The classical Brill-Noether theory aims to the description of all
families $G^r_d(C)$ of linear series of fixed degree $d$ and
dimension $r$ on a given curve $C$ of genus $g$. Equivalently, one
can consider the image via the Abel-Jacobi map $W^r_d(C) \subseteq
{\rm Pic}^d(C)$ of $G^r_d(C)$. In such a generality, the project
is certainly too ambitious. However, for $C$ sufficiently general
in $\mathcal M_g$ the problem has been completely solved. The main
results are Griffiths-Harris' theorem (see \cite{GH2}), which
determines the dimensions of the families $G^r_d(C)$, and
Gieseker's  theorem (see \cite{Gie}), proving the so called {\em
Petri's conjecture} which refines Griffiths-Harris' result giving
further important information about the local structure of
$G^r_d(C)$. Recall also Fulton-Lazarsfeld's theorem (see
\cite{FuLaz}) asserting that $W^r_d(C)$ is connected, for any
curve $C$, as soon as its dimension is positive.

\noindent Brill-Noether's type of questions can be asked replacing
line bundles with vector bundles. This can be done in various ways
(cf.\ e.g.\ \cite{BPL}, \cite{RT} and \cite{Te} as general
references).

In this paper we will take the following view-point. Given a curve
$C$ of genus $g \geq 1$, one can consider the moduli space
$U_C(d)$ of semistable, degree $d$, rank-two vector bundles on
$C$, which is an irreducible, projective variety of dimension
$4g-3 + \epsilon$, where $\epsilon = 1$, if $g = 1$ and $d$ is
even, $\epsilon = 0$ otherwise (cf.\ e.g.\ \cite{New}). For any
$[\Ff] \in U_C(d)$, one can consider the set
\begin{equation}\label{eq:intquot}
M_n(\Ff) := \{ N \subset \Ff \; | \; N \; \mbox{sub-line bundle of} \; \Ff, \; \deg(N) = n \},
\end{equation}which has a natural structure of Quot-scheme. Note that $M_n(\Ff)$ is
isomorphic to $M_{n+2l} (\Ff \otimes L)$, for any $L \in {\rm
Pic}^l(C)$. If $[\Ff] \in U_C(d)$ is general, then $M_n(\Ff)$ is
not empty if and only if $n \leq \left\lfloor
\frac{d-g+1}{2}\right\rfloor =: \overline{n}$ (cf.\ Corollary
\ref{thm:maruyama}, Remark \ref{rem:2allag} and \cite{Ma}). The
problem we consider here is to study the loci $M_n(\Ff)$, for $C$
general of genus $g$ and $\Ff$ general in $U_C(d)$, as well as
their  images $W_n(\Ff)$ in ${\rm Pic^n(C)}$. Of course, similar
questions can be asked for vector bundles of any rank and in this
generality they have been considered by various authors (see
\cite{LN}, \cite{Ma}, \cite{Oxb}).

As well known, the study of vector bundles on curves is equivalent
to the one of scrolls in projective space. Therefore, the above
questions can  be translated in terms of the geometry of scrolls.
Let $S$ be a smooth, non-special scroll of degree $d$ and
sectional genus $g \geq 0$ which is linearly normal in $\Pp^R$, $R
= d- 2g +1$. If $d \geq
2g+3+\min\{1,g-1\}$, such scrolls fill up a unique component
$\HH_{d,g}$ of the Hilbert scheme of surfaces in $\Pp^R$ which
dominates ${\mathcal M}_g$ (cf.\ Theorem \ref{thm:Lincei} below).

Let $[S] \in \HH_{d,g}$ be a general point, such that $S \cong
\Pp(\Ff)$, where $\Ff$ is a very ample rank-two vector bundle of
degree $d$ on $C$, a curve of genus $g$ with general moduli, and
$S$ is embedded in $\Pp^R$ via the global sections of
$\Oc_{\Pp(\Ff)} (1)$. In \cite{CCFMnonsp} we showed that, if $g
\geq 1$ and $S$ is general, then $\Ff$ is general in $U_C(d)$
(cf.\ \cite{APS} and \cite[Theorem 5.5]{CCFMnonsp}). We then proved that $S$ is a {\em
general ruled surface} in the sense of Ghione \cite{Ghio}, namely
the scheme ${\rm Div}_S^{1,m}$ parametrizing unisecant curves of
given degree $m$ on $S$ behaves as {\em expected} (for details,
cf.\ \cite[Def. 6.6 and Thm. 6.9]{CCFMnonsp}). If we put $m :=
d-n$, in Proposition \ref{prop:smmn} we prove that there is a
natural isomorphism $${\rm Div}_S^{1,m} \cong M_n(\Ff).$$This
provides the translation from the vector bundle to the scroll
setting. The map $$\pi_n: M_n(\Ff) \to W_n(\Ff) \subseteq {\rm
Pic}^n(C)$$can also be interpreted in terms of curves on the
scroll: the fibres of $\pi_n$ are connected (cf.\ Lemma
\ref{lem:fibres}) and can be identified with linear systems of
unisecant curves of degree $m$ on $S$. Therefore, the map $\pi_n$
can be regarded as an analogue of the Abel-Jacobi map. It is then
natural to consider the subschemes $W^r_n(\Ff) \subseteq W_n(\Ff)$
of points where the fibres of $\pi_n$ has dimension at least $r$.
These are analogues of the classical Brill-Noether loci.

The scheme ${\rm Div}_S^{1,m}$ was originally studied by C. Segre
(cf.\ \cite{Seg}), then in \cite{Ghio} and, by the present authors,
in \cite{CCFMnonsp}, where we used degeneration techniques. These
techniques, in particular the degeneration of a general scroll in
$\HH_{d,g}$ to the union of a rational normal scroll and $g$
quadrics (cf.\ Construction \ref{const:Y}), are also the main tool
in the present paper.

First of all, we apply the results in \cite{CCFMnonsp} to prove a conjecture
by Oxbury asserting that $M_{\overline{n}}(\Ff)$ is connected  for any curve $C$ of
genus $g$, $[\Ff] \in U_C(d)$ general and $d-g$ even
(cf.\ \cite[Conjecture 2.8]{Oxb}; Oxbury's conjecture refers more generally
to vector bundles of any rank).

Then, we turn to the consideration of  $W^r_n(\Ff)$.
In order to study such loci, a basic ingredient is the contraction map
$$\mu_N : H^0(\Ff \otimes N^{\vee}) \otimes H^0( \Ff^{\vee} \otimes \omega_C \otimes N) \to H^0(\omega_C)$$
which, in accordance to the line bundle case,
is called the {\em Petri map} of the pair $(\Ff, N)$ (cf.\ e.g.\ \cite{Oxb}).
The case $n = \overline{n}$ is  already studied in \cite{Oxb}
(cf.\ Proposition \ref{prop:pnbar} below). For $n < \overline{n}$ the situation is more complicated.
In Proposition \ref{prop:ganzo3}, we give some general results about the Brill-Noether
filtration in the general moduli case. In particular we show that, when $[S] \in \HH_{d,g}$ is general,
one has:
\begin{enumerate}[(a)]
\item if ${\rm dim}({\rm Div}_S^{1,m}) \geq g$ and  $[\Gamma] \in {\rm Div}_S^{1,m}$ is general,
then ${\rm dim}(|\Oc_S(\Gamma)|) = {\rm dim}({\rm Div}_S^{1,m})  - g,$
\item if $0 \leq {\rm dim}({\rm Div}_S^{1,m}) < g$ and  $[\Gamma] \in {\rm Div}_S^{1,m}$ is general, then
${\rm dim}(|\Oc_S(\Gamma)|) = 0.$
\end{enumerate}
In Theorem \ref{thm:existBN}, we concentrate on $W^1_n(\Ff)$ and, when $C$ has general moduli,
we prove that each of its irreducible components has the expected dimension.
We finish the paper by proving  an enumerative result, i.e.\ Theorem \ref{thm:sum}, in which  we
compute the class of the sum of all sub-line bundles of $\Ff$ of maximal degree, when these are
finitely many and $[\Ff] \in U_C(d)$ is general.

The paper is organized as follows.  In \S\ \ref{S:NotPre}
we collect standard definitions and properties of scrolls and unisecant curves.
In \S\ \ref{S:HNSS} we recall the results in \cite{CCFMLincei} and in \cite{CCFMnonsp}.
In \S's\ \ref{S:BrillNoether} and \ref{S:BN1np} we prove the above-mentioned
results of the Brill-Noether theory, whereas
\S\ \ref{SS:4} contains the enumerative result.

\section{Notation and preliminaries}\label{S:NotPre}

In this section we will fix notation and general assumptions as in \cite{CCFMnonsp}. For non reminded terminology
we refer the reader to \cite{Ha}, \cite{New}, \cite{Ses} and \cite{CCFMnonsp}.

Let $C$ be a smooth, projective curve of genus $g \geq 0$ and let $\rho \colon F \to C$
be a {\em geometrically ruled surface} on $C$, namely $F = \Pp(\Ff)$,
for some rank-two vector bundle, or locally free sheaf, $\Ff$ on $C$.
In this paper, we shall make the following:
\begin{ass}\label{ass:1}
We assume that $h^0(C, \Ff) = R+1$, for some $R \geq 3$,
that $|\Oc_F(1)|$ is base-point-free  and that
the corresponding morphism $\Phi: F \to \Pp^R$ is birational to its image.
\end{ass}
\noindent
We denote by $d$ the degree $\deg(\Ff) := \deg(\det(\Ff))$.

\begin{definition}\label{def:scroll} The surface $\Phi(F) :=S \subset \Pp^R$
is called a {\em scroll of degree $d$ and of (sectional) genus $g$},
and $S$ is called the {\em scroll determined by the pair} $(\Ff, C)$.
$S$ is smooth if and only if $\Ff$ is very ample; if
$S$ is singular, then $F$ is its minimal desingularization.
For any $x \in C$, let $f_x := \rho^{-1}(x) \cong \Pp^1$. The line $l_x := \Phi(f_x)$ is called a {\em ruling} of
$S$. Abusing terminology, the
family $\{ l_x \}_{ x \in C}$ is also called the {\em ruling} of $S$.
\end{definition}
\noindent For further details on ruled surfaces, we refer to
\cite{GH}, \cite[\S\,V]{Ha}, \cite{APS}, \cite{GP1}, \cite{GP2},
\cite{GP3}, \cite{Ghio}, \cite{GS}, \cite{GiSo}, \cite{Ma},
\cite{MN}, \cite{Na}, \cite{Seg} and \cite{Sev1}. If we denote by
$H$ the section of $\rho$ such  that $\Oc_F(H) = \Oc_F(1)$, then
${\rm Pic}(F) \cong \ZZ[\Oc_F(H)] \oplus \rho^*({\rm Pic}(C))$; if
$\underline{d}\in {\rm Div}(C)$,  we denote by $\underline{d} f$
the divisor $\rho^*(\underline{d})$ on $F$, where $f$ is the
general fibre of $\rho$. A similar notation will be used when
$\underline{d}\in {\rm Pic}(C)$. Thus, any element of ${\rm
Pic}(F)$ corresponds to a divisor on $F$ of the form $nH +
\underline{d} f$, for some $n \in \ZZ$ and $\underline{d}\in {\rm
Pic}(C)$.
\begin{definition}\label{def:unisec} Any curve $B \in |H + \underline{d} f|$ is called a {\em unisecant curve}
of $F$. Any irreducible unisecant curve $B$ of $F$ is smooth and is called a
{\em section} of $F$.
\end{definition}

There is a one-to-one correspondence between sections $B$ of $F$ and surjections
$\Ff \twoheadrightarrow L$, with $L =L_B$ a line bundle on $C$ (cf.\ \cite[\S\;V, Prop.\ 2.6 and 2.9]{Ha}).
Then, one has an exact sequence
\begin{equation}\label{eq:Fund}
0 \to N \to \Ff \to L \to 0,
\end{equation}where $N$ is a line bundle on $C$.
If $L =\Oc_C( \underline{m}) $, with  $\underline{m} \in {\rm Div}^m(C)$,
then $m = HB$ and $B \sim H + ( \underline{m} - \det(\Ff)) f $.
One has
\begin{equation}\label{eq:Ciro410b}
\Oc_B(B) \cong N^{\vee} \otimes L
\end{equation}
(cf.\ \cite[\S\,5]{Ha}). In particular,
\begin{equation}\label{eq:Ciro410}
B^2 = \deg(L) - \deg(N) = d - 2 \, \deg(N) = 2m - d.
\end{equation}
Similarly, if $B_1$ is a reducible unisecant curve of $F$ such that $H B_1 = m$,
there exists a section $B \subset F$
and an effective divisor $\underline{a} \in {\rm Div}(C)$,
$a:= \deg(\underline{a})$, such that $B_1 = B + \underline{a} f,$
where $BH = m-a$.
In particular there exists a
line bundle $L = L_B$ on $C$, with $\deg(L) = m-a$,
fitting in \eqref{eq:Fund}. Thus, one obtains the exact sequence
\begin{equation}\label{eq:Fund2}
0 \to N \otimes \Oc_C( - \underline{a}) \to \Ff \to L \oplus \Oc_{\underline{a}} \to 0.
\end{equation}(for details, cf.\ \cite{CCFMnonsp}).

\begin{definition}\label{def:direct} Let $S$ be a scroll of degree $d$ and
genus $g$ corresponding to $(\Ff, C)$ and  let $B \subset F$ be a section and
$L $ as in \eqref{eq:Fund}. If $\Phi|_B$ is birational to its image, then
$\Gamma:= \Phi(B)\subset S$ is called  a {\em section} of $S$. We will say that the pair
$(S,\Gamma)$ is {\em associated with} \eqref{eq:Fund} and that $\Gamma$ {\em corresponds to} $L$ on $C$.
If $m = \deg(L)$, then $\Gamma$ is a {\em section of degree} $m$
of $S$; moreover, $\Phi|_B : B \cong C \to \Gamma$
is determined by the linear series $\Lambda \subseteq |L|$, which
is the image of the map $H^0(\Ff) \to H^0(L).$
\noindent
If $B_1 \subset F$ is a (reducible) unisecant curve and $\Phi|_{B_1}$ is birational to its image,
then $\Phi(B_1) = \Gamma_1$ is a {\em unisecant curve of degree} $m$ of $S$. As above, the pair $(S, \Gamma_1)$
corresponds to a sequence of type \eqref{eq:Fund2}.
\end{definition}

By Riemann-Roch, one has
\begin{equation}\label{eq:R+}
R+1:= h^0(\Oc_F(1)) = d - 2 g + 2 + h^1(\Oc_F(1)).
\end{equation}
\begin{definition}[{cf.\ \cite[\S\;3, p.\ 128]{Seg}}]\label{def:spec}
We will call $h^1 (\Oc_F(1))$ the {\em speciality} of the scroll $S$.
A scroll $S$ is said to be {\em special} if $h^1(\Oc_F(1)) > 0$,
{\em non-special} otherwise.
\end{definition}

\noindent
For bounds  and remarks on $h^1(\Oc_F(1))$, we refer the reader to
e.g.\ \cite[Lemma 3.7, Example 3.10]{CCFMnonsp} and to \cite[pp.\ 144-145]{Seg}.
\begin{definition}\label{def:lndirec} Let $\Gamma_1 \subset S$ be a unisecant curve of $S$ of degree $m$
such that $(S,\Gamma_1)$ is associated to a sequence like \eqref{eq:Fund2}.
Then, $\Gamma_1$ is said to be
{\em special}, if $h^1(C, L)>0$, and {\em linearly normally embedded},
if $H^0(\Ff) \twoheadrightarrow H^0(L \oplus \Oc_{\underline{a}}).$
\end{definition}

\section{Hilbert schemes}\label{S:HNSS}
Let $S$ be a linearly normal, non-special scroll of degree $d$ and genus $g$.
When $g=0$, $S$ is rational and its properties are well-known
(see e.g.\ \cite{GH}). Thus, from now on, we shall focus on the case $g\geq 1$.
From \eqref{eq:R+}, one has that $S \subset \Pp^R$ where $R = d-2g + 1$ and $d\geq2g+2$,
because of the condition $R\geq3$ in Assumptions \ref{ass:1}. If, in addition, we assume that $S$ is smooth,
then $d\geq 2g+3+k$, where $k=\min\{1,g-1\}$ (cf.\ e.g.\ \cite[Remark 4.20]{CCFMLincei}).
In this situation, one has the following result essentially contained in \cite{APS} (cf. also 
\cite[Theorem 1.2]{CCFMLincei} and \cite[Theorem 5.4]{CCFMnonsp}).

\begin{theorem}\label{thm:Lincei}
Let $g \geq 0$ be an integer and let $k = {\rm min}\{1, g-1\}$. If
$ d \geq 2 g + 3 + k$, there exists a unique, irreducible component $\HH_{d,g}$ of the
Hilbert scheme of scrolls of degree $d$, sectional genus $g$ in
$\Pp^R$ such that the general point $[S] \in \HH_{d,g}$ represents a
smooth, non-special and linearly normal scroll $S$.
Furthermore,
\begin{enumerate}[{\rm(i)}]
\item $\HH_{d,g}$ is generically reduced;
\item ${\rm dim} (\HH_{d,g}) = 7(g-1) + (d-2g+2)^2 = 7(g-1) + (R+1)^2$;
\item $\HH_{d,g}$ dominates the moduli space
${\mathcal M}_g$ of smooth curves of genus $g$.
\end{enumerate}
\noindent
If, moreover $g \geq 1$, let $(\Ff, C)$ be a pair which determines $S$, where
$[C] \in {\mathcal M}_g$ is general. If $U_C(d)$ denotes the moduli space of semistable, degree $d$, 
rank-two vector bundles on $C$, then $[\Ff] \in U_C(d)$ is general.
\qed
\end{theorem}

We recall a construction of some reducible surfaces corresponding to points in $\HH_{d,g}$.
This is one of the key ingredients of the degeneration arguments used in \cite{CCFMnonsp},
which will also be used in this paper. The presence of points in $\HH_{d,g}$ corresponding to reducible 
surfaces was already pointed out in \cite{APS}. However te reducible surfaces we need in this paper 
are different. 

\begin{constr}[{see  \cite[Construction 5.11]{CCFMnonsp}}]\label{const:Y}
Let $g \geq 1$.  Then  $\HH_{d,g}$ contains points  $[Y]$ such that $Y$ is a reduced, connected, reducible surface, with global normal  crossings, of the form
\begin{equation}\label{eq:Y}
Y:= W \cup Q_1 \cup \cdots \cup Q_g,
\end{equation}where $W$ is a rational normal scroll, corresponding to a general point of $\HH_{d-2g,0}$,
each $Q_j$ is a smooth quadric, such that  $Q_j \cap Q_k = \emptyset$, if $1 \leq j \neq k \leq g$,
and $W \cap Q_j = l_{1,j} \cup l_{2,j}$, where  $l_{i,j}$ are general rulings of $W$, for $1 \leq i
\leq 2$, $1 \leq j \leq g$, and where the intersections are transverse.
Furthermore, for any such $Y$, one has that $h^1(Y, \N_{Y/\Pp^R}) = 0;$ in particular, $[Y]$ is a smooth point of $\HH_{d,g}$.
\end{constr}

We finish this section with the following definition and result.
\begin{definition}[{see \cite[Definition 6.1]{Ghio}}]\label{def:ghio0}
Let $C$ be a smooth, projective curve of genus $g
\geq 0$. Let $F = \Pp(\Ff)$ be a geometrically ruled surface over
$C$ and let $d = \deg(\Ff)$. For any positive integer $m$, we denote by
\begin{equation}
{\rm Div}_F^{1,m}
\end{equation}
the Hilbert scheme of unisecant curves of $F$,
which are of degree $m$ with respect to $\Oc_F(1)$; it has a natural structure as a Quot-scheme (cf.\ \cite{Groth}),
whose {\em expected dimension} is
\begin{equation}\label{eq:div1m}
d_m := {\rm max} \{ -1, \; 2 m - d - g + 1\};
\end{equation}therefore ${\rm dim}({\rm Div}_F^{1,m}) \geq d_m$.
\end{definition}
In \cite{CCFMnonsp}, we proved
\begin{theorem}[{see \cite[Theorem 6.9]{CCFMnonsp}}]\label{thm:ganzo2}
Let $g, d, \HH_{d,g}$ be as in Theorem \ref{thm:Lincei}.
If $[S] \in \HH_{d,g}$ is a general point, then $S$ is a
{\em general ruled surface}, namely, for any $m\geq1$:
\begin{enumerate}[(i)]
\item[(i)] $ {\rm dim}({\rm Div}_S^{1,m}) = d_m$, for any $m\geq1$;
\item[(ii)] ${\rm Div}_S^{1,m}$ is smooth, for any $m$ such that $d_m \geq 0$;
\item[(iv)] ${\rm Div}_S^{1,m}$ is irreducible, for any $m$ such that $d_m > 0$.
\qed
\end{enumerate}
\end{theorem}

\section{Brill-Noether theory}\label{S:BrillNoether}

\subsection{Preliminaries}\label{S:BN1}
Let $S \subset \Pp^R$ be a smooth, non degenerate scroll of degree $d$ and genus $g$.  Let
$(\Ff, C)$ be a pair determining $S$. Let $\Gamma$ be any unisecant curve of $S$ of degree $m$,
corresponding to the exact sequence
\begin{equation}\label{eq:Fund0}
0 \to N \to \Ff \to L \oplus \Oc_{\underline{a}} \to 0,
\end{equation}where $L$ and $N$ are line bundles and $\underline{a} \in {\rm Div}^a(C)$ such that $m = \deg(L) + a$.
Set
\begin{equation}\label{eq:n}
n := \deg(N) = d - m.
\end{equation}

In this section we will study the subschemes of ${\rm Pic}(C)$ parametrizing the sub-line bundles
$N \subset \Ff$ as in \eqref{eq:Fund0}.

\begin{definition}\label{def:seginv}
Let $C$ be a smooth, projective curve of genus $g \geq 0$ and let $\Ff$ be any
rank-two vector bundle on $C$. The {\em Segre
invariant} of $\Ff$ is defined as: $$s(\Ff) :=
\deg(\Ff) - 2 ({\rm Max} \; \{\deg (N) \}),$$where the maximum is taken
among all the sub-line bundles $N$ of $\Ff$ (cf.\ e.g.\ \cite{LN}).
We denote by $M(\Ff)$ the set of all sub-line bundles of $\Ff$
of maximal degree. Notice that $M(\Ff)$ has a natural structure of Quot-scheme
(cf.\ e.g.\ \cite{Oxb}).
\end{definition}In other words, $s(\Ff)$ is the minimum of the
self-intersections of sections
of  $F:= \Pp(\Ff)$ (cf.\ Formula \eqref{eq:Ciro410b} and
see e.g.\ \cite{LN}) and therefore, $ s(\Ff) = s(\Ff \otimes L)$, where $L$ is any line bundle.
Similarly, $M(\Ff)$ is isomorphic to $M(\Ff \otimes L)$.
Note that the vector bundle $\Ff$ is stable (resp., semi-stable) if and only if $s(\Ff) \geq 1$
(resp., $s(\Ff) \geq 0$).
In the following proposition we recall a result by Nagata, cf.\ \cite{Na}.
\begin{proposition}\label{prop:ciromarzo} Let $C$ be a smooth, projective curve of genus $g \geq 0$ and
let $\Ff$ be any rank-two vector bundle on $C$. One has:
\begin{equation}\label{eq:Nagata}
s(\Ff) \leq g.
\end{equation}
\end{proposition}
\begin{proof} Let $d = {\rm deg} (\Ff)$. Let
$\Gamma$ be a section of $F = \Pp(\Ff)$, such that $\Gamma^2 = s(\Ff)$. It corresponds
to an exact sequence of type \eqref{eq:Fund0}, with $\underline{a} = 0$. Let $m = \deg(L)$, so that
$\Gamma^2= 2m -d$ (cf.\ Formula \eqref{eq:Ciro410}).
Consider ${\rm Div}_F^{1,m}$. By the assumption $\Gamma^2 = s(\Ff)$, then all the curves in ${\rm Div}_F^{1,m}$
are sections. Therefore, ${\rm dim}({\rm Div}_F^{1,m}) \leq 1$. On the other hand, by \eqref{eq:div1m},
${\rm dim}({\rm Div}_F^{1,m}) \geq d_m = 2m - d - g + 1 = \Gamma^2 - g + 1$. Hence, \eqref{eq:Nagata} follows.
\end{proof}
The proof of Proposition \ref{prop:ciromarzo} shows that sub-line bundles $N$ of $\Ff$ with
maximal degree $\overline{n}$ correspond to sections in ${\rm Div}_F^{1,\overline{m}}$, with
$0 \leq d_{\overline{m}} \leq 1$.

\begin{lemma}\label{lem:2510} Let $C$ be a curve of genus $g \geq 1$
with general moduli and let $[\Ff] \in U_C(d)$ be a general point.
Then, the line bundles in $M(\Ff)$ have degree
\begin{equation}\label{eq:nbar}
\overline{n} :=  \left\lfloor \frac{d-g+1}{2}\right\rfloor.
\end{equation}
\end{lemma}
\begin{proof} From what recalled above on $M(\Ff)$, by tensoring $\Ff$ with sufficiently large multiple of an ample
line bundle we can assume that the scroll $S$ corresponding to the pair $(\Ff,C)$ is the general point in $\HH_{d,g}$ as in
Theorem \ref{thm:Lincei}. The assertion follows from Theorem \ref{thm:ganzo2} and from \eqref{eq:Fund}.
\end{proof}

Let $n$ be any integer such that
\begin{equation}\label{eq:boundn}
n \leq \overline{n}.
\end{equation}For any such $n$, one can consider the set
\begin{equation}\label{eq:quot}
M_n(\Ff) := \{ N \subset \Ff \; | \; N \; \mbox{sub-line bundle of} \; \Ff, \; \deg(N) = n \}.
\end{equation}With this notation,  $M_{\overline{n}}(\Ff) = M(\Ff)$ as in
Definition \ref{def:seginv} (cf.\ also \cite{Oxb}).
As  for the maximal case, any $M_n(\Ff)$ has a natural structure
of Quot-scheme.

For any $[N] \in  M_n(\Ff) $, one can
define $s_N(\Ff) := {\rm deg} (\Ff) - 2\;{\rm deg}(N) $; observe that,
as for the Segre invariant, one has $s_{N\otimes L} (\Ff \otimes L) =  s_N(\Ff) $, for any
$L \in Pic(C)$.
The proof of Lemma \ref{lem:2510} shows that,
in order to study the schemes $M_n(\Ff)$, for $C$ with general
moduli and $[\Ff] \in U_C(d)$ general, we may assume that the pair $(\Ff,C)$ determines a general point in
$\HH_{d,g}$ as in Theorem \ref{thm:Lincei}. Then, one has the morphism
\begin{equation}\label{eq:psimn}
\psi_{m,n} :  {\rm Div}_S^{1,m}  \to  M_n(\Ff),
\end{equation}with $m = d -n$ as in \eqref{eq:n}, defined by$$\psi_{m,n}([\Gamma]) = [N],$$where $\Gamma$ corresponds to
$L \oplus \Oc_{\underline{a}}$ on $C$ fitting  in \eqref{eq:Fund0}. The morphism $\psi_{m,n}$ is bijective;
in fact, given $N \hookrightarrow \Ff$ one has an exact sequence of type \eqref{eq:Fund0}, which uniquely
determines the corresponding unisecant curve $\Gamma$. This defines the inverse $\psi_{m,n}^{-1}$.
In particular, $${\rm dim} ({\rm Div}_S^{1,m}) = {\rm dim} (M_n(\Ff)).$$

\begin{proposition}\label{prop:smmn}
Let $g \geq 1$ and $d$ be integers as in Theorem \ref{thm:Lincei}. Let
$n \leq \overline{n}$ and $m = d-n$ be integers.
Let $[S] \in \HH_{d,g}$ be a general point. Then $$\psi_{m,n} :  {\rm Div}_S^{1,m}  \to  M_n(\Ff)$$is
an isomorphism.
\end{proposition}

\begin{proof} Since $M_n(\Ff)$ is a Quot-scheme, it is smooth at those points $[N] \in M_n(\Ff)$ such that
${\rm Ext}^1(N, \Ff/N) = (0)$. From \eqref{eq:Fund0}, ${\rm Ext}^1(N, \Ff/N) \cong
H^1( (L \oplus \Oc_{\underline{a}}) \otimes N^{\vee}) \cong H^1(L  \otimes N^{\vee}).$
Let  $[\Gamma_1] \in {\rm Div}_S^{1,m}$ be the unisecant curve as in \eqref{eq:Fund0}.
$\Gamma_1$ is of the form
$$\Gamma_1 = \Gamma \cup l_1 \cup \cdots \cup l_a, \;\; a = \deg(\underline{a}),$$where
$[\Gamma] \in {\rm Div}_S^{1,m-a}$ is a section and the $l_i$'s are lines of the ruling.
From the inclusion of schemes $\Gamma \subset \Gamma_1$,
we get
\begin{equation}\label{eq:surj}
L \oplus \Oc_{\underline{a}} \twoheadrightarrow L.
\end{equation}
Therefore, the section $\Gamma$ corresponds to a sequence
$$0 \to N' \to \Ff \to L  \to 0,$$
where $N'$ is a line bundle on $C$ of degree $n' = n + a$. Moreover,
from \eqref{eq:surj}, it follows that
\begin{equation}\label{eq:inj}
N \hookrightarrow N'.
\end{equation}Since $H^1(L \otimes N^{\vee}) \cong H^0( \omega_C \otimes L^{\vee} \otimes N)^{\vee}$, by \eqref{eq:inj}
we have
\begin{equation}\label{eq:inj2}
H^0( \omega_C \otimes L^{\vee} \otimes N) \hookrightarrow H^0( \omega_C \otimes L^{\vee} \otimes N').
\end{equation}From \eqref{eq:Ciro410b},
$$L \otimes (N')^{\vee} \cong \N_{\Gamma/S}$$and
$h^1(\N_{\Gamma/S}) = 0$, for any $[\Gamma] \in {\rm Div}_S^{1,m-a}$ (cf.\ Theorem \ref{thm:ganzo2}).
This implies that $H^1(L \otimes N^{\vee}) = (0)$ so $M_n(\Ff)$ is smooth.
Since $\psi_{m,n}$ is bijective, it is an isomorphism (cf.\ \cite[Exercise I, 3.3]{Ha}).
\end{proof}
As an immediate consequence of Proposition \ref{prop:smmn} and of the
proof of Lemma \ref{lem:2510}, we have the following:
\begin{corollary}\label{cor:2510b} Let  $C$ be a curve of genus $g \geq 1$ with general moduli and
$[\Ff] \in U_C(d)$ be a general point. Let $ n \leq \overline{n}$ and $m = d-n$ be integers. Then
$ M_n(\Ff)$ is smooth, of dimension $d_m$ and it is irreducible when $d_m >0$.
\end{corollary}

Moreover, we have the following result (cf.\ \cite{Ma}, \cite[Corollary 3.2]{LN} and
\cite[Proposition 1.4, Theorem 3.1, Example 3.2]{Oxb}).

\begin{corollary}\label{thm:maruyama} Let $C$ be a smooth, projective  curve of genus
$g \geq 1$ and let $\Ff$ be a rank-two vector bundle of degree $d$ on $C$.
One has:
\begin{itemize}
\item[(a)] if $s(\Ff) = g$, then $ {\rm dim} (M(\Ff)) = 1$, $d - g$ is even and $\overline{n} = \frac{d-g}{2}$.
\item[(b)] if $C$ has general moduli, $[\Ff] \in U_C(d)$ general and  $s(\Ff) \leq g-1$, then
${\rm dim}(M(\Ff)) =0$, $s(\Ff) = g-1$, $d - g$ is odd and $\overline{n} = \frac{d-g +1 }{2}$.
\end{itemize}
\end{corollary}
\begin{proof} As usual, we may assume that $(\Ff,C)$ corresponds to a point in $\HH_{d,g}$.
Let $\overline{m} := d - \overline{n}$.

\noindent
(a) One has $ 1 \geq {\rm dim} (M(\Ff)) = {\rm dim}({\rm Div}_F^{1,\overline{m}}) \geq
d_{\overline{m}} =  1$ (cf.\ the proof of Proposition \ref{prop:ciromarzo}). The assertion follows.
\noindent
(b) By the generality assumptions, one has $ {\rm dim} (M(\Ff)) =
{\rm dim}({\rm Div}_F^{1,\overline{m}}) \geq
d_{\overline{m}} =  0$. The assertion follows.
\end{proof}

The following corollary proves a particular case of \cite[Conjecture 2.8]{Oxb}.
\begin{corollary}\label{rem:conjecture} Let $C$ be any smooth, projective curve of genus
$g \geq 1$. Let $d$ be an integer  such that $d-g$ is even. Let
$[\Ff] \in U_C(d)$ be general. Then $M(\Ff)$ is a
connected curve.
\end{corollary}
\begin{proof} By Corollary \ref{cor:2510b}, $M(\Ff)$ is a smooth and irreducible  curve if  $C$ has general moduli.
On the other hand, since we are in in case $(a)$ of Corollary \ref{thm:maruyama}, then
$M(\Ff)$ is in any case a curve. Now, \cite[Theorem 3.1]{Oxb} implies that the numerical equivalence
class of $M(\Ff)$  is independent on $C$. Therefore, $M(\Ff)$, as a limit of a smooth, irreducible curve,
is connected.
\end{proof}

\begin{remark}\label{rem:2allag}
Note that, in case $(b)$ of Corollary \ref{thm:maruyama}, Maruyama proves more, i.e.\ he assumes $C$ to be any
curve, $d$ any positive integer and $[\Ff] \in U_C(d)$ general. Furthermore, in this case,
$ M(\Ff)$ consists of $2^g$ distinct elements (cf.\ \cite[Theorem 16]{Tu},
\cite[Corollary 3.2]{LN} and \cite[Proposition 1.4, Theorem 3.1, Example 3.2]{Oxb}; see also
Proposition \ref{prop:smmn} and \cite[Theorem 7.1.1]{CCFMnonsp}).
Thus, when $d - g$ is odd, one has a rational map
\[
\lambda\colon U_C^s(d) \dashrightarrow Sym^{2^g}({\rm Pic}^{\frac{d-g}{2}}(C)).
\]
For $g =1$, $\lambda$ is everywhere defined and it is an isomorphism (cf.\ e.g.\ \cite[Remark 5.5]{CCFMnonsp}).
As soon as $g \geq 2$, ${\rm dim}(U_C(d)) < {\rm dim}( Sym^{2^g}({\rm Pic}^{\frac{d-g}{2}}))$. Natural questions are:
\begin{itemize}
\item is $\lambda$ injective?
\item is $d\lambda$ injective  where $\lambda$ is defined?
\end{itemize}Affirmative answers would give (global and infinitesimal) Torelli type theorems.
\end{remark}

\subsection{The Brill-Noether loci}\label{S:BN2}

As in \cite[\S\,1]{Oxb}, for any $n \leq \overline{n}$
one can consider the natural morphism
\begin{equation}\label{eq:pn}
\pi_n : M_n(\Ff) \to {\rm Pic}^n(C)
\end{equation}sending any sub-line bundle $N \subset \Ff$ of degree $n$ to $[N] \in {\rm Pic}^n(C)$. We shall denote by
\begin{equation}\label{eq:wn}
W_n(\Ff) := Im(\pi_n) \subseteq {\rm Pic}^n(C)
\end{equation}(cf.\ \cite[Theorem 3]{GhioBu}, \cite{GhioBrill} and \cite{Oxb}, where $W_{\overline{n}}(\Ff)$ is
denoted by $W(\Ff)$).
The map $\pi_n$ can be viewed as an analogue of the classical Abel-Jacobi map and
$M_n(\Ff)$ has to be viewed as an analogue of the symmetric product of the curve $C$.

\begin{lemma}\label{lem:fibres}
For any $[N] \in W_n(\Ff)$,$$\pi_n^{-1}([N]) \cong \Pp(H^0(\Ff \otimes N^{\vee})).$$In particular,
$\pi_n$ has connected fibres.
\end{lemma}
\begin{proof} This follows from the definition of $W_n(\Ff)$ (cf.\ \cite[p.\ 11]{Oxb}, for $n = \overline{n}$).
Indeed, $[N] \in W_n(\Ff)$ iff $[N] \in {\rm Pic}^n(C)$ is a sub-line bundle of $\Ff$, equivalently, iff there exists
a non-zero global section in $H^0(\Ff \otimes N^{\vee})$.
\end{proof}
\begin{remark}\label{rem:ganzo3} {\normalfont Recalling
\eqref{eq:psimn}, we have the commutative diagram
\begin{equation}\label{eq:diaganzo}
\xymatrix{{\rm Div}_S^{1,m} \ar[r]^{\psi_{m,n}} \ar[dr]!UL_{\Phi_{m,n}}
  & M_n(\Ff) \ar[d]^{\pi_n} \\
  & W_n(\Ff)
}%
\end{equation}
For any $[N] \in W_n(\Ff)$ and  $[\Gamma_N] = \psi_{m,n}^{-1}(N)$,
we have
\begin{equation}\label{eq:linsys}
\Pp(H^0(\Ff \otimes N^{\vee})) \cong |\Oc_S(\Gamma_N)|,
\end{equation}i.e., the fibres of $\pi_n$ can be identified with linear systems of unisecant curves of degree
$m = d-n$ on $S$.
}
\end{remark}

The above setting suggests the definition of Brill-Noether type loci in $W_n(\Ff)$. One proceeds
as follows. For any integer $p \geq 0$,
one defines the {\em Brill-Noether locus}
\begin{equation}\label{eq:wnp}
W_n^p(\Ff) := \{ [N] \in {\rm Pic}^n(C) \, | \, h^0(\Ff \otimes N^{\vee}) \geq p+1 \}.
\end{equation}Since $[\Ff] \in U_C(d)$ is general, this is a degeneracy-locus of  a suitable vector bundle map
on ${\rm Pic}^n(C)$ and, as such, has a natural scheme structure
(cf.\ the construction in \cite[pp.\ 11-12]{Oxb}, for the case $n = \overline{n}$,
which extends to any $n \leq \overline{n}$).
In particular, for any $n \leq \overline{n}$, $W(\Ff) = W^0_n(\Ff)$ and there is a filtration
\begin{equation}\label{eq:filtration}
\emptyset = W^{k+1}_n(\Ff)\subset W^k_n(\Ff) \subseteq W^{k-1}_n(\Ff) \subseteq \cdots
\subseteq W^2_n(\Ff) \subseteq W^1_n(\Ff) \subseteq W^0_n(\Ff) = W_n(\Ff),
\end{equation} for some $k \geq 0$ (cf.\ \cite{GhioBrill}). Note that, for any
$p \geq 0$, $W^{p+1}_n(\Ff)$ is contained in the singular locus of $W^p_n(\Ff)$.
Recalling Remark \ref{rem:ganzo3}, we see that the pull-back via $\Phi_{m,n}$ of $W_n^p(\Ff) $ is
\begin{equation}\label{eq:div1mp}
{\rm Div}_S^{1,m}(p) := \{ [\Gamma] \in {\rm Div}_S^{1,m} \;\; | \;\;\, {\rm dim}(|\Oc_S(\Gamma)|) \geq p \},
\end{equation}which is a subscheme of ${\rm Div}^{1,m}_S$ (cf.\ \cite[p.\ 68]{GhioBu}).
Via the isomorphism $\psi_{m,n}$, the scheme ${\rm Div}^{1,m}_S(p)$ can be identified with
\begin{equation}\label{eq:iso2}
M_n^p(\Ff) := \{ N \subset \Ff \, | \, \deg(N) = n \; {\rm and} \; h^0(\Ff \otimes N^{\vee}) \geq p+1 \},
\end{equation}which is the subscheme of $M_n(\Ff)$ pull-back of $W_n^p(\Ff)$ via $\pi_n$.

We recall the following proposition from \cite[Theorems 2, 3]{GhioBu}, \cite{GhioBrill} (see also
\cite[Lemma 2.2]{Oxb}, for the case $n = \overline{n}$):
\begin{proposition}\label{prop:classes}
Let $d_m $ be as in \eqref{eq:div1m}. For any integer $p \geq 0$, let
\begin{equation}\label{eq:taup}
\tau_p(\Ff) := {\rm max} \{ -1, \; g - (p+1) (p + g - d_m)\}.
\end{equation}If $W_n^p(\Ff) \neq \emptyset$, then
\begin{equation}\label{eq:ghiobrill}
{\rm dim} (W_n^p(\Ff)) \geq \; {\rm min}\,\{g, \tau_p(\Ff)\},
\end{equation}where the right-hand-side is the
{\em expected dimension} of $W_n^p(\Ff)$. In particular, with $d$ as in Theorem \ref{thm:Lincei}, one has:
\begin{itemize}
\item[(i)] if $0 \leq \tau_p(\Ff) < g$, then $${\rm dim} ({\rm Div}_S^{1,m}(p))  \geq \tau_p(\Ff) + p =:
{\rm expdim} ( {\rm Div}_S^{1,m}(p)),$$whereas,
\item[(ii)]  if $\tau_p(\Ff) = g$, then for any $p_0 \leq p$, one has
$$W_n^{p_0}(\Ff) = {\rm Pic}^n(C) \; {\rm and} \;
{\rm Div}_S^{1,m}(p_0) = {\rm Div}_S^{1,m};$$furthermore, the general fibre of $\Phi_{m,n}$ has dimension
$d_m - g = 2m - d - 2g + 1$.
\end{itemize}If, moreover, the equality in \eqref{eq:ghiobrill} holds with $0 \leq \tau_p(\Ff) < g$,
then the class in ${\rm Pic}^n(C)$ of $W_n^p(\Ff)$ is
\[
[W_n^p(\Ff)] \equiv \left(\prod_{i=0}^p \frac{i!}{(p+g+i-d_m)!}\right) \cdot 2 ^{g - \tau_p(\Ff)} \cdot\theta^{p - \tau_p(\Ff)},
\]where $\equiv$ denotes the numerical equivalence of cycles
and $\theta$ is the class of the theta divisor in ${\rm Pic}^n(C)$.
\qed
\end{proposition}
Note that, since $m = d-n$, one has
\begin{equation}\label{eq:tau0}
\tau_0(\Ff) = d_m ,
\end{equation}which agrees with the notion of expected dimension  for ${\rm Div}^{1,m}_S$
(cf.\ Formula \eqref{eq:div1m}). Moreover, in case (ii), for any
$[\Gamma] \in {\rm Div}^{1,m}_S$ one has$${\rm dim}(|\Oc_S(\Gamma)|) \geq 2m - d - 2g +1,$$which agrees
with Riemann-Roch theorem. Equality holds if ${\rm Div}^{1,m}_S$ has the expected dimension and
$[\Gamma] \in {\rm Div}^{1,m}_S$ is general.
For the proof of Proposition \ref{prop:classes}, see \cite{GhioBrill}.
In \cite[Theorems 2, 3]{GhioBu}, one finds
the expression of the class of $[{\rm Div}_S^{1,m}(p)]$ in ${\rm Div}^{1,m}_S$,
for $S$ a general ruled surface (cf.\ Theorem \ref{thm:ganzo2}).
In order to study the morphism
\begin{equation}\label{eq:pn2}
\pi_n: M_n(\Ff) \to W_n(\Ff)
\end{equation}and the schemes $W_n^p(\Ff)$, for $p \geq 0$, a basic ingredient is the following
contraction map
\begin{equation}\label{eq:petri}
\mu_N : H^0(\Ff \otimes N^{\vee}) \otimes H^0( \Ff^{\vee} \otimes \omega_C \otimes N) \to H^0(\omega_C)
\end{equation}defined for any $[N] \in M_n(\Ff)$. In accordance to the classical case of line bundles,
$\mu_N$ is called the {\em Petri map} of the pair $(\Ff, N)$
(cf.\ e.g.\ \cite{Oxb}). As in \cite[Ch.\ IV, \S 1]{ACGH}, one has
(cf.\ \cite[Prop.\ 2.4]{Oxb}, for the maximal case
$n = \overline{n}$):
\begin{lemma}\label{lem:ACGH}
For $[N] \in W_n^p(\Ff) \setminus W^{p+1}_n(\Ff),$
$$T_{[N]}(W_n^p(\Ff) ) \cong Im(\mu_N)^{\perp}.$$Therefore, if not empty, $W_n^p(\Ff) $ is smooth and
of the expected dimension at $[N]$ if and only if the Petri map $\mu_N$ is injective.
\end{lemma}
\noindent
Therefore if the Petri map $\mu_N$ is injective for any  $[N] \in W_n^p(\Ff) \setminus W^{p+1}_n(\Ff)$,
then the singular locus of $W_n^p(\Ff)$ coincides with $W^{p+1}_n(\Ff)$.

The maximal case $n = \overline{n}$ has been studied in \cite{Oxb}. We recall the results.
\begin{proposition}\label{prop:pnbar}
Let $g \geq 1$ be an integer and let
$C$ be any smooth, projective curve of genus $g$. For any integer $d$, let $[\Ff] \in
U_C(d)$ be general. Then:
\begin{itemize}
\item[(i)] the map $\pi_{\overline{n}}$  is an isomorphism; in particular,
$W_{\overline{n}}(\Ff)$ is smooth and strictly contained in ${\rm Pic}^{{\overline{n}}}(C)$.
\item[(ii)] $W^p_{\overline{n}}(\Ff) = \emptyset$, for any $p \geq 1$.
\item[(iii)] If $d$ is as in Theorem \ref{thm:Lincei} and if
\begin{equation}\label{eq:mbar}
\overline{m} := d - \overline{n} = \lfloor \frac{d+g}{2} \rfloor,
\end{equation}then for the general
$[S] \in \HH_{d,g}$ and for any $[\Gamma] \in {\rm Div}^{1,\overline{m}}_S$, one has
${\rm dim}(|\Oc_S(\Gamma)|) = 0.$
\end{itemize}
\end{proposition}
\noindent
Parts $(i)$ and $(ii)$ are contained in \cite{Oxb}. Part $(iii)$ is an immediate consequence
of Remark \ref{rem:ganzo3}.

\begin{remark}\label{rem:pnbar} {\normalfont $W_{\overline{n}}(\Ff)$
is a divisor in ${\rm Pic}^{\overline{n}}(C)$ when $g =2$ and
$d$ is even (cf.\ \cite[Remark 1.6]{Oxb}): up to twists, $d = 0$ so
$\overline{n}= - 1$; in this case, $W_{\overline{n}}(\Ff)$ can be identified with the divisor
$D_{\Ff} = \{ M \in {\rm Pic}^1(C) \; | \; h^0(\Ff \otimes M) = 1\} \in | 2 \Theta|$,
where $\Theta$ denotes the theta divisor in ${\rm Pic}^1(C)$.
}
\end{remark}

For $n < \overline{n}$ the situation is more complicated. We will prove the following:
\begin{proposition}\label{prop:ganzo3}
Let $C$ be a smooth, projective curve of genus $g\geq 1$ with general moduli and let $d$ be an integer.
Let $[\Ff] \in U_C(d)$ be general and let $\tau_0(\Ff) $ be as in
\eqref{eq:tau0}.  Let $n < \overline{n}$ be any integer.
\noindent
(a) If $\tau_0(\Ff) \geq g$, then for general $[N] \in M_n(\Ff)$, $h^1(\Ff \otimes N^{\vee}) = 0$ and we have the
filtration
\begin{equation}\label{eq:filtrcasa03}
\emptyset \subset \cdots \subseteq W^{d_m-g + 1}_n(\Ff) \subset W^{d_m-g }_n(\Ff) = \cdots
= W^{1}_n(\Ff)  = W_n(\Ff) = {\rm Pic}^n(C).
\end{equation}

\noindent
(b) If $ 0 \leq \tau_0(\Ff) < g$, then
$W_n^0(\Ff)$ is not empty, strictly contained in ${\rm Pic}^n(C)$ and
also the inclusion $W^1_n(\Ff) \subset W_n(\Ff)$ is strict. Moreover:
\begin{itemize}
\item[(i)] $W^0_n(\Ff)$ is smooth, of dimension $\tau_0(\Ff)$,
at any $[N] \in W^0_n(\Ff) \setminus W^1_n(\Ff)$.
\item[(ii)] $\pi_{n|} : M_n(\Ff) \setminus M_n^1(\Ff) \to  W^0_n(\Ff) \setminus W^1_n(\Ff)$ is an isomorphism.
\item[(iii)] $W^0_n(\Ff)$ is irreducible when $\tau_0(\Ff)>0$.
\end{itemize}
\end{proposition}
\begin{proof} $(a)$ As usual, we may assume that the pair $(\Ff,C)$ determines a general point in $\HH_{d,g}$ as in
Theorem \ref{thm:Lincei}. Consider the exact sequence
\begin{equation}\label{eq:casa02}
0 \to \Oc_C \to \Ff \otimes N^{\vee} \to (L \oplus \Oc_{\underline{a}}) \otimes N^{\vee} \to 0,
\end{equation}obtained from \eqref{eq:Fund0}. One has
$h^1((L \oplus \Oc_{\underline{a}}) \otimes N^{\vee} ) = 0$ (see the proof of Proposition \ref{prop:smmn}).
Thus
\begin{equation}\label{eq:ganza3}
0 \to H^0(\Oc_C) \to H^0(\Ff \otimes N^{\vee}) \to H^0(L \otimes N^{\vee}) \stackrel{\partial}{\to}
H^1(\Oc_C) \to H^1(\Ff \otimes N^{\vee}) \to 0,
\end{equation}where the coboundary map $\partial$ can be identified with the differential of the morphism
$\pi_n : M_n(\Ff) \to {\rm Pic}^n(C)$. Since $\tau_0(\Ff) \geq g$, the morphism
$\pi_n$ is surjective (cf.\ Proposition \ref{prop:classes} - $(ii)$).
Hence $\partial$ is surjective if $[N]$ is general and therefore
$h^1(\Ff \otimes N^{\vee}) = 0$.

\noindent
$(b)$ Since $\tau_0(\Ff) = d_m$, as in \eqref{eq:tau0}, then from Theorem \ref{thm:ganzo2}
${\rm dim} ({\rm Div}^{1,m}_S) = d_m  \geq 0$. By \eqref{eq:psimn},
also $M_n(\Ff) \neq \emptyset$, so $W^0_n(\Ff)$ is not empty. Since ${\rm dim}(M_n(\Ff)) = d_m$, by
\eqref{eq:ghiobrill}, we have ${\rm dim} (W^0_n(\Ff)) = \tau_0(\Ff) = d_m$ (cf.\ also \cite[Theorem 0.3]{RT}).
Since $\tau_0(\Ff) < g$, then $W^0_n(\Ff)$ is strictly contained in ${\rm Pic}^n(C)$.
From Proposition \ref{prop:smmn} and Lemma \ref{lem:fibres}, it follows that
$\pi_n: M_n(\Ff) \to W_n(\Ff)$ is birational. Since $M_n(\Ff)$ is smooth (see Theorem \ref{thm:ganzo2}
and Proposition \ref{prop:smmn}), then the scheme $W_n(\Ff)$ is generically smooth. This proves that
the inclusion $W^1_n(\Ff) \subset  W_n(\Ff)$ is strict.
\noindent
Part $(i)$ follows
by the injectivity of the Petri map $\mu_N$. In fact, $h^0(\Ff \otimes N^{\vee}) = 1$, for
any $[N] \in W_n(\Ff) \setminus W_n^1(\Ff)$. Moreover,
since $[\Ff] \in U_C(d)$ is general, then $\Ff$ is {\em very-stable}
(cf.\ \cite{Lau} and \cite[p.\ 12]{Oxb}), which means that $\mu_N$ is injective on each factor of the tensor product.

\noindent
Part $(ii)$ follows since $\pi_n|$ is a bijective morphism between smooth varieties hence it is an isomorphism.

\noindent
Part $(iii)$ follows from Theorem \ref{thm:ganzo2} and Proposition \ref{prop:smmn}.
\end{proof}

The above argument shows the following:

\begin{corollary}\label{cor:ganzo3}
Let $d$ and $g$ be positive integers as in Theorem \ref{thm:Lincei}.
Let $C$ be a smooth, projective curve of genus $g$ with general moduli.
Let $[\Ff] \in U_C(d)$ be general.  Let $[S] \in \HH_{d,g}$
be determined by $(\Ff, C)$. Let $m > \overline{m}$ be any integer. Then:
\begin{itemize}
\item[(a)]  If $d_m \geq g$ and $[\Gamma] \in {\rm Div}_S^{1,m}$ is general,
then ${\rm dim}(|\Oc_S(\Gamma)|) = \tau_0(\Ff) - g = d_m - g.$

\item[(b)] If  $0 \leq d_m < g$, then for any unisecant curve
$[\Gamma] \in {\rm Div}_S^{1,m} \setminus {\rm Div}^{1,m}_S(1)$ one has
${\rm dim}(|\Oc_S(\Gamma)|) = 0.$
\end{itemize}
\end{corollary}

In the circle of ideas presented in this section, a natural and interesting problem would be to prove
the analogue of Petri's conjecture:
\begin{conjecture}\label{conj:petri} {\em Let $C$ be a smooth, projective curve of genus $g$ with general moduli.
Let $[\Ff] \in U_C(d)$ be general. Let $[N] \in M_n(\Ff)$ be any point. Then, the Petri map
$\mu_N$ is injective.}
\end{conjecture}As remarked above, the validity of this conjecture would imply:
\begin{itemize}
\item[(i)] $W_n^p(\Ff)$ has the expected dimension, i.e.\ ${\rm min} \{g, \; \tau_p(\Ff)\}$ as in \eqref{eq:ghiobrill};
\item[(ii)] $W_n^p(\Ff)$ is smooth off $W_n^{p+1}(\Ff)$.
\end{itemize}Statement $(i)$ above is an analogue of the Brill-Noether Theorem.
In the next section, we will prove $(i)$ for $p=1$ under suitable numerical assumptions.

\section{Brill-Noether's theorem for $W^1_n(\Ff)$}\label{S:BN1np}

In this section we will study $W^1_n(\Ff)$ and prove that it has the expected dimension $e := e^1_n(d)$,
which is:
\begin{itemize}
\item[(i)] $-1$, when $n > \frac{2d-3g}{4}$,
\item[(ii)] $2d - 4n - 3g < g$, when $\frac{d-2g}{2} < n \leq \frac{2d-3g}{4}$,
\item[(iii)] $g$, when $n \leq \frac{d-2g}{2}$,
\end{itemize}(cf.\ \eqref{eq:taup}, \eqref{eq:ghiobrill}). Case $(iii)$ is contained in Proposition
\ref{prop:ganzo3}-$(a)$. Therefore, it suffices to consider $n>\frac{d-2g}{2}$.

\begin{theorem}\label{thm:existBN} Let $C$ be a smooth, projective curve of genus $g \geq 1$ with general moduli and
$d$ be an integer. Let $[\Ff] \in U_C(d)$ be general. Let $n > \frac{d-2g}{2}$ be any integer.
Then, each irreducible component of $W^1_n(\Ff)$ has the expected
dimension.
\end{theorem}
\begin{proof} As usual, we can assume that the pair $(\Ff,C)$ corresponds to a general point $[S] \in \HH_{d,g}$.
In order to prove the theorem, it suffices to show that, for $m = d-n$, one has
${\rm dim} ({\rm Div}_S^{1,m}(1)) = e +1$, if $e \geq 0$,  whereas ${\rm Div}_S^{1,m}(1)$ is empty, if $e = -1$
(cf.\ \eqref{eq:div1mp}). We will prove this by degeneration,
studying the limit of ${\rm Div}_S^{1,m}(1)$ when $S$ degenerates to a surface $Y = W \cup Q_1 \cup \cdots
\cup Q_g$, where $W$ is a general rational normal scroll of degree $d-2g$ and $Q_1, \ldots, Q_g$ are
general quadrics as in Construction \ref{const:Y}, from which we keep the notation.

In order to study the limit in question, let $\Pm$ be a linear
pencil of curves in ${\rm Div}_S^{1,m}$ and let $\Pm_0$ be the flat limit of $\Pm$ on $Y$. Then
$\Pm_0$ consists of a collection of linear pencils $\LL$, $\LL_1, \ldots, \LL_g$ of unisecant curves on
$W$, $Q_1, \ldots, Q_g$. By the genericity of $Q_1, \ldots, Q_g$ none of these pencils
contain the double lines $l_{i,j}$, where $ 1\leq i \leq 2$, $ 1 \leq j \leq g$, in their fixed locus.
Moreover, they verify the obvious matching properties along them.
Let $\mu$, $\mu_1, \ldots, \mu_g$ be the degrees
of the curves in $\LL$, $\LL_1, \ldots, \LL_g$, respectively. We will call such a $\Pm_0$ a
{\em limit unisecant pencil} of type $(\mu, \mu_1, \ldots, \mu_g)$.
One has $m = \mu + \sum_{i=1}^g \mu_i$.
We may assume that $\mu_1 = \mu_2 = \ldots = \mu_h = 1$, whereas $\mu_{h+1}, \ldots , \mu_g \geq 2$. Note that
$h \geq 1$; otherwise we would have $m \geq \mu + 2g \geq \frac{d-2g}{2} + 2g = \frac{d}{2} + g$ (cf.\ \eqref{eq:mbar}
applied to $\mu$ and $W$). This reads $d \geq 2n + 2g$ which implies $\tau_1(\Ff) \geq g$ hence $ e= g$,
a case which we are not considering.
Recall that $W \cap Q_j$ consists of the pair of lines $l_{1,j}$, $l_{2,j}$, $1 \leq j \leq g$. The
Segre embedding $\Sigma_j$ of $l_{1,j} \times l_{2,j}$ sits in a $\Pp^3$, whose dual we denote by $\Pi_j$.
Let $\G$ be the grassmannian of lines in ${\rm Div}^{1,\mu}_W$. One has a natural rational map
$$r\colon \G \dashrightarrow \Pi_1 \times \ldots \times \Pi_h,$$
which is defined as follows. Let
$\LL$ be a general pencil in ${\rm Div}^{1,\mu}_W$; $\LL$ cuts on the divisor
$ l_{1,j} + l_{2,j}$ a linear series of dimension one and degree two which can be interpreted
as a curve on $\Sigma_j$, cut out by a plane corresponding to a point $\ell_j \in
\Pi_j$. The map $r$ sends $\LL$ to the $h$-tuple $(\ell_1, \ldots, \ell_h)$.

\begin{claim}\label{cl:cilib1} If $e = -1$, then $r$ is not dominant.
\end{claim}
\begin{proof}[Proof of Claim \ref{cl:cilib1}] One has $ m = \mu + h + \sum_{j=h+1}^g \mu_j \geq
\mu + 2g - h $. The assumption $e = -1$ is equivalent to $m < \frac{2d+3g}{4}$;
therefore one has $\frac{2d+3g}{4} > \mu + 2g - h $, i.e.\ $4 \mu + 5g - 2d < 4 h $, which implies
${\rm dim}(\G) = 4 \mu + 4g - 2d < 3h = {\rm dim} (\Pi_1 \times \ldots \times \Pi_h)$. This proves
the assertion.
\end{proof}
This claim settles the case $e=-1$. In fact it shows that,
by the genericity of the quadrics $Q_1, \ldots, Q_h$, the pencils
$\LL_1, \ldots , \LL_h$ cannot match any pencil $\LL$ on $W$ to give a limit unisecant pencil $\Pm_0$.
Thus, from now on, we assume $e \geq   0$ and we study the possible components of the
flat limit of ${\rm Div}_S^{1,m}(1)$ when $S$ degenerates to $Y$. Since ${\rm Div}_S^{1,m}(1)$ is
not empty in this case, its flat limit  is not empty. Let
$\Pm_0$ be a limit unisecant pencil of type $(\mu, \mu_1,\ldots,\mu_g)$ as above.
By the genericity of the quadrics $Q_1,\ldots,Q_h$, the map $r$ has to be dominant.
Let $\Psi$ be the general fibre of $r$. One has
\begin{equation}\label{eq:cilib1}
\dim(\Psi)=\dim(\G)-3h=4m - 2d -3g- \sum_{i=h+1}^g (4\mu_i-7) .
\end{equation}
Now we are ready to compute the dimension of a component of limit unisecant
pencils.  Let $\G_j$ be the grassmannian of lines of ${\rm Div}^ {1,\mu_j}_{Q_j}$, for $j=h+1,\ldots,g$.
We have two  rational maps
\[
p\colon \Psi \dashrightarrow \Pi_{h+1}\times \ldots\times \Pi_g,
\qquad
q\colon \G_{h+1}\times \dots \times \G_g \dashrightarrow \Pi_{h+1}\times \ldots\times \Pi_g
\]defined as follows. A general point of $\Psi$ is a pencil $\LL$ in ${\rm Div}_W^{1,\mu}$. It cuts a
linear series of degree 2
and dimension 1 on the divisor  $ l_{1,j} + l_{2,j}$, $j=h+1,\ldots,g$, which, as usual, gives rise to a point
$\ell_j \in \Pi_j$. The map $p$ sends $\LL$ to $(\ell_{h+1}, \ldots, \ell_g)$. The definition of the map $q$
is similar (see the proof of Claim \ref {cl:cilib4}  below). A component $Z$ of limit unisecant  pencils of
type $(\mu,\mu_1,\dots,\mu_g)$ can be interpreted as an irreducible component of the fibred product of $p$ and $q$.
\begin{claim}\label{cl:cilib4} All fibres of the map $q$
have dimension $\sum_{i=h+1}^g(4 \mu_i - 7)$.
\end{claim}
\begin{proof} Fix a $j=h+1,\ldots,g$. We have a map $q_j: \G_j \dashrightarrow \Pi_j$ and $q= q_{h+1}\times
\ldots \times q_g$. It suffices to prove that all fibres of $q_j$ have dimension $4\mu_j-7$.
Since $\mu_j>1$, the linear system ${ {\rm Div}^ {1,\mu_j} }_{Q_j}$ of dimension $2\mu_j-1$ cuts out a
complete linear series $\Lambda_j$ of dimension 3 on the divisor $ l_{1,j} + l_{2,j}$. So we have a surjective projection map
$s_j\colon { {\rm Div}^ {1,\mu_j} }_{Q_j} \dashrightarrow \Lambda_j$,
with centre a projective space of dimension $2\mu_j-5$.
This induces a map $\sigma_j\colon \G_j \dashrightarrow \overline {\G}_j$, where
$\overline {\G}_j$ is the grassmannian of lines of $\Lambda_j$.
All fibres of $\sigma_j$ are grassmannians of dimension $4\mu_j-8$.
We have also a map $\tau_j: \overline {\G}_j \dashrightarrow \Pi_j$ sending, as usual,
a pencil in $\Lambda_j$ to a point  $\ell_j \in \Pi_j$. Every fibre of $\tau_j$ has dimension 1.
Indeed, a point $\ell_j \in \Pi_j$ can be interpreted as a projective transformation
$\omega_j:  l_{1,j} \to  l_{2,j}$, and this in turn determines the quadric $\Omega_j$ described by all
lines joining corresponding points on $ l_{1,j}$ and $ l_{2,j}$. The pairs of such
corresponding  points are cut out by all pencils of planes based on lines of the ruling of
$\Omega_j$ to which $ l_{1,j}$ and $ l_{2,j}$ belong.
Since $q_j=\tau_j\circ \sigma_j$, the above considerations imply the assertion.
\end{proof}

Putting together \eqref{eq:cilib1} and Claim \ref{cl:cilib4}, one obtains that
${\rm dim}(Z) = 2d - 4 n - 3g$, which proves the theorem.
\end{proof}

\section{Tensor product of quotient line bundles}\label{SS:4}
In this section we consider the following problem.
Let $C$ be a smooth, projective curve with general moduli and $d$ be an integer.
Let $[\Ff] \in U_C(d)$ be general. Assume $d - g$ odd, and
let $\overline{n} = \frac{d - g +1}{2}$ as in \eqref{eq:nbar}. Let
$[N_i ] \in M_{\overline{n}} (\Ff)$, with
$N_i \neq N_j$, for $ 1 \leq i \neq j \leq 2^g$,
and let $\nu_i$ denote a divisor class on $C$ such that
$N_i = \Oc_C(\nu_i)$. We want to compute the equivalence class of the
divisor $$\nu := \sum_{i=1}^{2^g} \nu_i.$$Set
$L_i := \det(\Ff) \otimes N_i^{\vee}$ and let $\lambda_i$ be a divisor class
such that $L_i = \Oc_C(\lambda_i)$. Consider
$\lambda:= \sum_{i=1}^{2^g} \lambda_i$ and notice the relation
$\lambda + \nu = 2^g \, H$, where $\det(\Ff) = \Oc_C(H)$.

\begin{theorem}\label{thm:sum} With the above notation, one has:
\begin{equation}\label{eq:sumetai}
\nu = 2^{g-2} (2H - K_C), \; {\rm if} \; g \geq 2, \;\;{\rm and}\;\; \; \nu = H , \; {\rm if} \; g=1
\end{equation}and
\begin{equation}\label{eq:sumlambdai}
\lambda = 2^{g-2} (2H + K_C), \; {\rm if} \; g \geq 2, \;\;{\rm and}\;\; \; \lambda = H , \; {\rm if} \; g=1.
\end{equation}
\end{theorem}

\begin{proof} It suffices to show \eqref{eq:sumetai}.
\begin{claim}\label{claim:ciroduke} There exist $\alpha, \; \beta \in \ZZ$ such that
\begin{equation}\label{eq:ciroduke}
\nu = \alpha K_C + \beta H.
\end{equation}
\end{claim}We first show that Claim \ref{claim:ciroduke} implies \eqref{eq:sumetai}. Then, we will prove the claim.

If $g=1$, $K_C$ is trivial and therefore the first summand in \eqref{eq:ciroduke}
does not appear. Moreover, it is clear that $\alpha$ and $\beta$ in \eqref{eq:ciroduke} do not depend on $H$.
As usual, we may assume that the pair $(\Ff,C)$ is associated to a scroll $S$ correpsonding to a general
point in $\HH_{d,g}$. Since any $N_i$ is a maximal sub-line bundle of $\Ff$, then each $L_i$ corresponds to a section in
${\rm Div}_S^{1,\overline{m}}$, where $\overline{m} = d - \overline{n} = \frac{d+g-1}{2}$.
Consider the exact sequence
$$0 \to N_i \to \Ff \to  L_i \to 0,\;\;\; 1 \leq i \leq 2^g.$$Let $p \in C$
be a general point. Twist the above sequence by $\Oc_C(p)$
$$0 \to N_i(p) \to \Ff (p) \to  L_i(p) \to 0,\;\;\; 1 \leq i \leq 2^g.$$Observe that
$$H':= {\rm det} ((\Ff(p)) = H \otimes \Oc_C(2p),$$hence
${\rm deg} (\Ff(p)) = d+2$ and $\Ff(p)$ corresponds to a general point of $U_C(d+2)$. Thus,
$N_i(p)$  is a sub-line bundle of $\Ff(p)$ of maximal degree, for $1 \leq i \leq 2^g$. Set
$N_i(p) = \Oc_C(\nu_i')$ and
$\nu' = \sum_{i=1}^{2^g} \nu_i'$. One has
\begin{equation}\label{eq:ott12}
\nu' = \nu + 2^g p.
\end{equation}
By Claim \ref{claim:ciroduke}, there exist two integers $\alpha', \beta'$, independent on $H$ and $p$,
such that $$\nu'= \alpha'K_C + \beta'H'.$$By comparing the former relation with \eqref{eq:ciroduke} and
\eqref{eq:ott12}, we find
$$(\alpha - \alpha') K_C + (\beta - \beta')H + 2 (2^{g-1} - \beta') p = 0.$$Since
$\alpha, \alpha', \beta, \beta' \in \ZZ$ do not depend on $H$ and $p$, we deduce
$$\beta  = \beta' = 2^{g-1} \;\;\; {\rm and} \;\;\; \alpha = -2^{g-2},$$proving \eqref{eq:sumetai}
(in case $g=1$, we simply get $\nu = H$).
\medskip

We are left to prove Claim \ref{claim:ciroduke}. To do this, we follow a similar argument as in \cite{Ciro}.
Let ${\mathcal M}^0_g$ be the Zariski open subset of the moduli space ${\mathcal M}_g$, whose points
correspond to equivalence classes of smooth curves of genus $g$ without non-trivial automorphisms.
By definition, ${\mathcal M}^0_g$ is a fine moduli space, i.e.\ we have a universal family
$p : {\mathcal C} \to {\mathcal M}^0_g$, where ${\mathcal C}$ and ${\mathcal M}^0_g$ are smooth schemes and
$p$ is a smooth morphism. ${\mathcal C} $ can be identified with the Zariski open
subset ${\mathcal M}^0_{g,1}$ of the moduli space ${\mathcal M}_{g,1}$ of smooth, pointed, genus $g$
curves, whose points correspond to equivalence classes of pairs $(C,x)$, with $x \in C$ and $C$ a smooth
curve of genus $g$ without non-trivial automorphisms. On  ${\mathcal M}^0_{g,1}$ there is again a universal family
$p_1 : {\mathcal C}_1 \to {\mathcal M}^0_{g,1}$, where
${\mathcal C}_1 = {\mathcal C} \times_{{\mathcal M}^0_g} {\mathcal C}$. The family $p_1$ has a natural regular
global section $\delta$ whose image is the diagonal. By means of $\delta$, for any integer $n$, we have
the universal family of Picard varieties of order $n$,
i.e.\ $$p_1^{(n)} : {\mathcal Pic}^{(n)} \to {\mathcal M}^0_{g,1}$$
(cf.\ \cite[\S\,2]{Ciro}). For any closed point
$[(C,x)] \in {\mathcal M}^0_{g,1}$, its fibre via $p_1^{(n)}$ is isomorphic to ${\rm Pic}^{(n)}(C)$.
As in \cite[Theorem 5.4]{CCFMnonsp}, let $q: {\mathcal U}_{d} \to {\mathcal M}^0_{g,1}$
be the relative moduli stack of degree $d$, rank-two semistable vector bundles; namely,
the fibre of $q$ over $[(C,x)]$ is $U_C(d)$. Thus, we have the following natural surjective map
over ${\mathcal M}^0_{g,1}$
\begin{equation}\label{eq:reldet}
{\mathcal U}_{d} \stackrel{rd}{\to}  {\mathcal Pic}^{(d)},
\end{equation}which is given by the {\em relative determinant}; namely
$$rd((C, x, \Ff)) = (C, x, \det(\Ff)),$$for any $[(C,x)] \in {\mathcal M}^0_{g,1}$ and any
$[\Ff] \in U_C(d)$.
Observe that the fibre of $rd$ over any $[(C,x, H)] \in {\mathcal M}^0_{g,1}$ is
$SU_C(H)$, i.e.\ the moduli stack of semistable, rank-two vector
bundles on $C$ with fixed determinant $H \in {\rm Pic}^d(C)$. From \cite{Ball}, we know that
any $SU_C(H)$ is {\em stably rational}, i.e.\ $SU_C(H) \times \Pp^k$ is rational
for some $k \geq 0$. In particular, it is unirational.

Set $a := \deg (\nu) = 2^{g-1} (d-g+1)$. One has an obvious morphism
\begin{equation}\label{eq:varphi}
\varphi:  {\mathcal U}_{d}  \to {\mathcal Pic}^{(a)}
\end{equation}which maps $(C,x, \Ff)$ to the class of $\nu$. By the unirationality of the
fibres of $rd$, we have a morphism $\phi$ which makes the following diagram
commutative
\begin{equation}\label{eq:diag}
\xymatrix{ {\mathcal U}_{d} \ar[r]^{rd} \ar[dr]!UL_{\varphi}
  &  {\mathcal Pic}^{(d)}  \ar[d]^{\phi} \\
  &  {\mathcal Pic}^{(a)}.
}
\end{equation}
At this point, one concludes by imitating the proof of \cite[Proposition (5.1)]{Ciro},
which can be repeated almost verbatim.
\end{proof}



\begin{thebibliography}{ADSE}


\bibitem{ACGH} E.~Arbarello, M.~Cornalba, P.A.~Griffiths, J.~Harris,  {\em Geometry of algebraic curves, Vol. I}.
Grundlehren der Mathematischen Wissenschaften, {\bf 267}, Springer-Verlag, New York, 1985.

\bibitem{APS} E.~Arrondo, M.~Pedreira, I.~Sols, On regular and stable ruled surfaces in $\Pp^3$,
{\em Algebraic curves and projective geometry} (Trento, 1988),  1--15. With an appendix of
R.~Hernandez, 16--18, Lecture Notes in Math., {\bf 1389}, Springer-Verlag, Berlin, 1989.

\bibitem{Ball} E.~Ballico,
Stable rationality for the variety of vector bundles over an algebraic curve,
{\em J. London Math. Soc.}, {\bf 30} (1984), no. 1, 21--26.

\bibitem{BPL} L.~Brambila-Paz, H.~Lange,
A stratification of the moduli space of vector bundles on curves.
Dedicated to Martin Kneser on the occasion of his 70th birthday.
{\em J. Reine Angew. Math.},  {\bf 494} (1998), 173--187.

\bibitem{CCFMLincei} A.~Calabri, C.~Ciliberto, F.~Flamini, R.~Miranda,
Degenerations of scrolls to unions of planes, {\em Rend. Lincei Mat. Appl.}, {\bf 17} (2006), no.2, 95--123.

\bibitem{CCFMnonsp} A.~Calabri, C.~Ciliberto, F.~Flamini, R.~Miranda,
Non-special scrolls with general moduli, preprint.

\bibitem{Ciro} C.~Ciliberto, On rationally determined line bundles on a familly of projective
curves with general moduli, {\em Duke Math. J.}, {\bf 55} (1987), no. 4, 909--917.

\bibitem{FuLaz} W.~Fulton, R.~Lazarsfeld,
On the connectedness of degeneracy loci and special divisors,
{\em Acta Math.},  {\bf 146} (1981), no. 3-4, 271--283.



\bibitem{GP1} L.~Fuentes-Garcia, M.~Pedreira, Canonical geometrically ruled surfaces,
{\em Math. Nachr.}, {\bf 278} (2005), no. 3, 240--257.

\bibitem{GP2} L.~Fuentes-Garcia, M.~Pedreira, The projective theory of ruled surfaces,
{\em Note Mat.}, {\bf 24} (2005), no. 1, 25--63.

\bibitem{GP3} L.~Fuentes-Garcia, M.~Pedreira, The general special scroll of genus $g$ in $\Pp^N$. Special scrolls in $\Pp^3$, math.AG/0609548 (2006), pp. 13.

\bibitem{Ghio} F.~Ghione, Quelques r\'esultats de Corrado Segre sur les surfaces r\'egl\'ees,
{\em Math. Ann.}, {\bf 255} (1981), 77--95.

\bibitem{GhioBu} F.~Ghione, La conjecture de Brill-Noether pour les surfaces r\'egl\'ees,
{\em Proceedings of the Week of Algebraic Geometry} (Bucharest, 1980), 63--79, Teubner-Texte zur Math., {\bf 40}, Teubner, Leipzig, 1981.

\bibitem{GhioBrill} F.~Ghione, Un probl\`eme du type Brill-Noether pour les fibr\'es vectoriels,
{\em Algebraic geometry---open problems (Ravello, 1982)}, 197--209, Lecture Notes in Math., {\bf 997}, Springer, Berlin, 1983.

\bibitem{GS} F.~Ghione, G.~Sacchiero, Genre d'une courbe lisse trac\'ee sur une vari\'et\'e r\'egl\'ee,
{\em Space curves} (Rocca di Papa, 1985), 97--107, Lecture Notes in Math., {\bf 1266}, Springer, Berlin, 1987.

\bibitem{GiSo} L.~Giraldo, I.~Sols, The irregularity of ruled surfaces in $\Pp^3$. Dedicated to the memory
of Fernando Serrano.  {\em Collect Math.}, {\bf 49} (1998), no. 2--3, 325--334.


\bibitem{Gie} D.~Gieseker, Stable curves and special
divisors: Petri's conjecture. {\em Invent. Math.}, {\bf 66}
(1982), no. 2, 251--275.

\bibitem{GH2} P.~Griffiths, J.~Harris, On the
variety of special linear systems on a general algebraic curve.
{\em Duke Math. J.}, {\bf 47} (1980), no. 1, 233--272.

\bibitem{GH} P.~Griffiths, J.~Harris, {\em Principles of Algebraic Geometry},
Wiley Classics Library, New York, 1978.


\bibitem{Groth} A.~Grothendieck, Technique de descente et th\'eor\`eme d'existence en g\'eom\'etrie alg\'ebrique.
V {\em S\'eminaire Bourbaki}, {\bf 232}, 1961--1962.

\bibitem{Ha} Hartshorne R., {\em Algebraic Geometry}, Graduate Text in Math., {\bf 52},
Springer-Verlag, New York, 1977.


\bibitem{LN} H.~Lange, M.S.~Narashiman, Maximal subbundles of rank two vector bundles on curves,
{\em Math. Ann.}, {\bf 266} (1983), 55--72.



\bibitem{Lau} G.~Laumon, Un analogue global du c\^one nilpotent,
{\em Duke Math. J.}, {\bf 57} (1989), no. 2, 647--671.

\bibitem{Ma} A.~Maruyama, On classification of ruled surfaces,
{\em Lectures in Mathematics. Kyoto University}, no. {\bf 3}, Tokyo, 1970.


\bibitem{MN} A.~Maruyama, M. Nagata, Note on the structure of a ruled surface,
{\em J. reine angew. Math.}, {\bf 239} (1969), 68--73.



\bibitem{Na} M.~Nagata, On self-intersection number of a section on a ruled surface,
{\em Nagoya Math. J.}, {\bf 37} (1970), 191--196.

\bibitem{New} P.E.~Newstead, {\em Introduction to moduli problems and orbit spaces},
Tata Institute of Fundamental Research Lectures on Mathematics and Physics, {\bf 51},
Narosa Publishing House, New Delhi, 1978.

\bibitem{Oxb} W.M.~Oxbury, Varieties of maximal line bundles,
{\em Math. Proc. Camb. Phil. Soc.}, {\bf 129} (2000), 9--18.


\bibitem{RT} B.~Russo, M.~Teixidor I Bigas, On a conjecture of Lange,
{\em J. Algebraic Geom.}, {\bf 8} (1999), 483--496.

\bibitem{Seg} C.~Segre, Recherches g\'en\'erales sur les courbes et les surfaces r\'egl\'ees alg\'ebriques,
{\em OPERE - a cura dell'Unione Matematica Italiana e col contributo del Consiglio Nazionale delle Ricerche},
{\bf vol. 1}, \S XI - pp. 125--151, Edizioni Cremonese, Roma 1957
(cf.\ {\em Math. Ann.}, {\bf 34} (1889), 1--25).


\bibitem{Ses} C.S.~Seshadri, {\em Fibr\'es vectoriels sur les courbes alg\'ebriques},
Ast\'erisque, {\bf 96}. S.M.F., Paris, 1982.


\bibitem{Sev1} F.~Severi, Sulla classificazione delle rigate algebriche, {\em Univ. Roma e Ist. Naz. Alta Mat.
Rend. Mat. e Appl.}, {\bf 2} (1941), 1--32.



\bibitem{Te} M.~Teixidor i Bigas, Brill-Noether theory for stable vector bundles,
{\em Duke Math. J.}, {\bf 62} (1991), no. 2, 385--400.



\bibitem{Tu} L.W.~Tu, Semistable bundles over an elliptic curve. {\em Adv. Math.}, {\bf 98} (1993), no. 1, 1--26.

\end{thebibliography}
\end{document}